\documentclass[11pt]{amsart}

\newtheorem{theorem}{Theorem}[section]
\newtheorem{lemma}[theorem]{Lemma}

\theoremstyle{definition}

\newtheorem{corollary}[theorem]{Corollary}
\newtheorem{remark}[theorem]{Remark}
\theoremstyle{remark}

\newcommand{\be}{\begin{equation}}
\newcommand{\ee}{\end{equation}}

\numberwithin{equation}{section}

\begin{document}

\title{The rigidity of Dolbeault-type operators and symplectic circle actions }

\author{Ping Li}
\address{Department of Mathematics, Tongji University, Shanghai 200092, China}
%\curraddr{Department of Mathematics and Statistics, Case Western
%Reserve University, Cleveland, Ohio 43403}
\email{pingli@tongji.edu.cn}
\thanks{Research supported by Program for Young Excellent
Talents in Tongji University.}

%\author{Author Two} \address{Mathematical Research Section, School
%of Mathematical Sciences, Australian National University, Canberra
%ACT 2601, Australia} \email{two@maths.univ.edu.au}

%    General info

%\date{January 1, 2001 and, in revised form, June 22, 2001.}


\begin{abstract}
Following the idea of Lusztig, Atiyah-Hirzebruch and Kosniowski, we
note that the Dolbeault-type operators on compact, almost-complex
manifolds are rigid. When the circle action has isolated fixed
points, this rigidity result will produce many identities concerning
the weights on the fixed points. In particular, it gives a criterion
to detemine whether or not a symplectic circle action with isolated
fixed points is Hamiltonian. As applications, we simplify the proofs
of some known results related to symplectic circle actions, due to
Godinho, Tolman-Weitsman and Pelayo-Tolman, and generalize some of
them to more general cases.
\end{abstract}

\maketitle

\section{Introduction}
This paper is concerned with the rigidity of Dolbeault-type
operators on compact, almost-complex manifolds and its relationship
to symplectic circle actions.

Let $(M^{2n},\omega)$ be a compact, connected symplectic manifold of
dimension $2n$. A circle action is called \emph{symplectic} if it
preserves the symplectic form $\omega$ or, equivalently, if the one
form $\omega(X,\cdot)$ is closed, where $X$ is the generating vector
field of this circle action. This symplectic circle action is called
\emph{Hamiltonian} if $\omega(X,\cdot)$ is exact, i.e.,
$\omega(X,\cdot)=df$ for some smooth function $f$ on $M$. We call
$f$ the \emph{moment map} of this action, which is unique up to a
constant.

An obvious necessary condition for a symplectic circle action to be
Hamiltonian is to have \emph{non-empty} fixed points corresponding
to the critical points of $f$ (the minimum and the maximum of $f$
must be fixed points). In the case of K\"{a}hler manifolds \cite{Fr}
and of four-dimensional symplectic manifolds \cite{Mc}, this
condition is also sufficient. However, this is not true for general
higher dimensional symplectic manifolds. In fact McDuff \cite{Mc}
constructed a six-dimensional symplectic manifold with a symplectic
circle action which has fixed points, but which is not Hamiltonian.

Note that the fixed point sets of the counterexamples in \cite{Mc}
are tori. Hence one possible conjecture is that a symplectic action
with \emph{isolated} fixed points must be Hamiltonian. Some
partially affirmative results have been obtained: Tolman-Weitsman
\cite{TW} showed that this conjecture holds for semifree circle
actions (an action is called \emph{semifree} if it is free outside
the fixed points); Godinho \cite{Go} showed that this holds on
certain circle actions on six-dimensional symplectic manifolds with
specified weights on the fixed points.

McDuff \cite{Mc} showed that a symplectic circle action is
Hamiltonian \emph{if and only if} there is a connected component of
the fixed point set such that all the weights of the representation
of the circle action on the normal bundle of the component are
positive. Therefore in order to show that a symplectic circle action
is actually Hamiltonian, it suffices to show that there exists a
connected component of the fixed point set satisfying this
condition. The main ideas of the proof in \cite{TW} and \cite{Go}
are both based on this argument. In the isolated fixed points case,
Fel'dman \cite{Fe} refined McDuff's observation to show that the
Todd genus of a manifold admitting a symplectic circle action with
isolated fixed points is either 0, in which case the action is
non-Hamiltonian, or 1, in which case the action is Hamiltonian.

In a recent paper of Pelayo and Tolman \cite{PT}, the authors showed
that if a compact symplectic manifold admits a symplectic (not
necessarily Hamiltonian) circle action with isolated fixed points,
then the weights on the fixed points must satisfy some restrictions
(\cite{PT}, Theorem 2). We would like to point out that Theorems 1
and 3 in \cite{PT} are closely related to some much earlier work of
Kosniowski (\cite{Ko2}, \cite{Ko3}). Theorem 1 in \cite{PT} is
related to a conjecture of Kosniowski (\cite{Ko3}, p.338, Conjecture
A). Theorem 3 in \cite{PT} has been obtained in (\cite{Ko2}, Theorem
2) for complex manifolds and in (\cite{Ko3}, p.337) in the more
general case. The forms of the weights on the two fixed points are
hidden in the last paragraph of \cite{Ko2}.

Our paper is inspired by the above-mentioned results. This paper is
organized as follows. In Section 2, we consider some Dolbeault-type
elliptic operators on compact, almost-complex manifolds. If an
almost-complex manifold admits a circle action compatible with the
almost-complex structure, we can define the equivariant indices of
these operators under this circle action. Then following an idea of
Lusztig, Atiyah-Hirzebruch and Kosniowski, we will prove the
invariance of the  relevant equivariant indices of these
Dolbeault-type operators under circle actions having isolated fixed
points. This is the meaning of the word \emph{rigidity} in our
title. When an almost-complex manifold admits a compatible circle
action with isolated fixed points, this rigidity result immediately
produces many identities concerning the weights on the fixed points.
In particular, combining with Fel'dman's result \cite{Fe}, we give a
criterion to determine whether or not a symplectic circle action is
Hamiltonian.

In Section 3, as the first application of our result, we give a
simple and unified new proof of Godinho's result \cite{Go}, of which
the original proof is very complicated (see the whole Section 3 of
\cite{Go}). In fact our conclusion is more general than that of
Godinho. As the second main application, we also generalize Pelayo
and Tolman's above-mentioned result to circle actions on
almost-complex manifolds.

\begin{remark}
Given an elliptic operator, we say it is \emph{rigid} if, under
\emph{any} circle action, the corresponding equivariant index of
this operator is invariant. A survey of the results concerning the
rigidity of some important elliptic operators can be found in
Section $1$ of \cite{BT}.
\end{remark}
\section{The rigidity of Dolbeault-type operators and main results}
Let $(M^{2n},J)$ be a compact, almost-complex manifold of real
dimension $2n$ with an almost-complex structure $J$. The choice of
an almost Hermitian metric on $M$ enables us to define the Hodge
star operator $\ast$ and the formal adjoint
$\bar{\partial}^{\ast}=-\ast\bar{\partial}~\ast$ of the
$\bar{\partial}$-operator (\cite{GH}, p.80). Then for each $0\leq
p\leq n$, there is an elliptic differential operator (cf. \cite{LM},
p.258, Example 13.14)
\be\label{GDC}\bigoplus_{q~\textrm{even}}\Omega^{p,q}(M)\xrightarrow{\bar{\partial}+\bar{\partial}^{\ast}}\bigoplus_{q
~\textrm{odd}}\Omega^{p,q}(M),\ee where
$\Omega^{p,q}(M):=\Gamma(\Lambda^{p}T^{\ast}M\otimes\Lambda^{q}\overline{T^{\ast}M}).$
Here $T^{\ast}M$ is the dual of holomorphic tangent bundle $TM$ in
the sense of $J$. The index of this operator is denoted by
$\chi^{p}(M)$
($=\textrm{dim}_{\mathbb{C}}\textrm{ker}(\bar{\partial}+\bar{\partial}^{\ast})-\textrm{dim}_{\mathbb{C}}\textrm{coker}(\bar{\partial}+\bar{\partial}^{\ast})$)
in the notation of Hirzebruch \cite{Hi}. We define the Hirzebruch
$\chi_{y}$-genus $\chi_{y}(M)$ by
$$\chi_{y}(M)=\sum_{p=0}^{n}\chi^{p}(M)\cdot y^{p}.$$
\begin{remark}\begin{enumerate}
\item When
$J$ is integrable, i.e., $M$ is a $n$-dimensional complex manifold,
$\chi^{p}(M)$ equals to the index of the following well-known
Dolbeault complex
\be\label{DC}0\rightarrow\Omega^{p,0}(M)\xrightarrow{\bar{\partial}}\Omega^{p,1}(M)\xrightarrow{\bar{\partial}}\cdots\xrightarrow{\bar{\partial}}\Omega^{p,n}(M)\rightarrow
0\ee and hence $\chi^{p}(M)=\sum_{q=0}^{n}(-1)^{q}h^{p,q}(M),$ where
$h^{p,q}(M)$ is the corresponding Hodge numbers of $M$.

\item
For a general almost-complex manifold, $\bar{\partial}^{2}$ is not
identically zero (it is a well known fact that
$\bar{\partial}^{2}\equiv 0$ is equivalent to the integrality of
$J$). So we cannot define the Dolbeault complex (\ref{DC}).
Therefore (\ref{GDC}) may be taken as the Dolbeault-type complex in
the almost-complex case.

\item
Using the general form of the Riemann-Roch-Hirzebruch theorem (first
proved by Hirzebruch for projective manofolds \cite{Hi}, and in the
general case by Atiyah and Singer \cite{AS}), we have
$$\chi^{p}(M)=<\textrm{ch}(\Lambda^{p}T^{\ast}M)\textrm{td}(TM),[M]>,$$
where $\textrm{ch}(\cdot)$ is the Chern character and
$\textrm{td}(\cdot)$ is the Todd class. $[M]$ is the fundamental
class of $M$ induced from $J$ and $<\cdot,\cdot>$ is the Kronecker
pairing.
\end{enumerate}
\end{remark}

Now suppose $M$ admits a circle action ($S^{1}$-action) preserving
the almost-complex structure $J$. Then for any $g\in S^{1}$, we can
define the equivariant index $\chi^{p}(g,M)$ (resp. equivariant
Hirzebruch $\chi_{y}$-genus
$\chi_{y}(g,M):=\sum_{p=0}^{n}\chi^{p}(g,M)y^{p}$) of the elliptic
operators in (\ref{GDC}) by choosing an invariant almost Hermitian
metric under this circle action. Note that $\chi^{p}(g,M)$ is a
\emph{finite} Laurent series in $g$ as both
$\textrm{ker}(\bar{\partial}+\bar{\partial}^{\ast})$ and
$\textrm{coker}(\bar{\partial}+\bar{\partial}^{\ast})$ are
finite-dimensional.

Moreover, we assume the fixed points of this action are non-empty
and isolated, say $P_{1},\cdots,P_{m}$. At each $P_{i}$, there are
well-defined $n$ integer weights $k^{(i)}_{1},\cdots,k^{(i)}_{n}$
(not necessarily distinct) induced from the isotropy representation
of this $S^{1}$-action on $T_{p_{i}}M$. Note that these
$k^{(i)}_{1},\cdots,k^{(i)}_{n}$ are \emph{nonzero} as the fixed
points are isolated. We use $e_{i}(x_{1},\cdots,x_{n})$ ($1\leq
i\leq n$) to denote the $i$-th elementary symmetric polynomial of
$x_{1},\cdots,x_{n}$ and $e_{0}(x_{1},\cdots,x_{n}):=1$. With these
notations understood, the main observation in this paper is the
following
\begin{theorem}\label{main theorem}
Let $(M^{2n},J)$ be a compact, connected, almost-complex manifold
with a circle action preserving the almost-complex structure $J$.
Suppose the fixed points of this action are non-empty and isolated.
Let the notations be as above. Then for each $0\leq p\leq n$, the
expression
$$\sum_{i=1}^{m}\frac{e_{p}(g^{k^{(i)}_{1}},\cdots,g^{k^{(i)}_{n}})}{\prod_{j=1}^{n}(1-g^{k_{j}^{(i)}})}$$
is a constant and equals to $\chi^{p}(M)$. Here $g$ is an
indeterminate.
\end{theorem}
Note that $\chi^{0}(M)$ is nothing else but the Todd genus of $M$.
Hence combining with Fel'dman's observation (\cite{Fe}, Theorems 1
and 2) mentioned in the Introduction, we have
\begin{theorem}
If $M$ is a symplectic manifold and the circle action is symplectic,
then the
expression
$$\sum_{i=1}^{m}\frac{1}{\prod_{j=1}^{n}(1-g^{k_{j}^{(i)}})}$$
is either $0$, in which case the action is non-Hamiltonian, or $1$,
in which case the action is Hamiltonian.
\end{theorem}

Our proof follows the exposition of Section 5.7 in \cite{HBJ}.
Although the considerations of Sections 5.6 and 5.7 in \cite{HBJ}
are for compact complex manifolds, we will see that, when replacing
the Dolbeault complexes (\ref{DC}) by the elliptic complexes
(\ref{GDC}), the proof can also be applied to the situation of
almost-complex manifolds with no more difficulties.

\emph{Proof of Theorem \ref{main theorem}.}

Let $g\in S^{1}$ be a topological generator, i.e., the closure of
$\{g^{r}~|~r\in\mathbb{Z}\}$ is the whole $S^{1}$. Then the fixed
point set of the action $g$ is exactly $\{P_{1}, \cdots,P_{m}\}$. We
note that the characteristic power series corresponding to the
Hirzebruch $\chi_{y}$-genus is $x\frac{1+ye^{-x}}{1-e^{-x}}$. Then
the Lefschetz fixed point formula of Atiyah-Bott-Segal-Singer
(\cite{AS}, p.256) tells us that
\be\label{LOC}\chi_{y}(g,M)=\sum_{i=1}^{m}\prod_{j=1}^{n}\frac{1+yg^{k^{(i)}_{j}}}{1-g^{k^{(i)}_{j}}}.\ee
Note that the left hand side (LHS) of (\ref{LOC}) is a \emph{finite}
Laurent series in $g$. Hence the only possible singularities are $0$
and $\infty$. While \be\label{lim}\lim_{g\rightarrow
0}\prod_{j=1}^{n}\frac{1+yg^{k^{(i)}_{j}}}{1-g^{k^{(i)}_{j}}}=(-y)^{d_{i}},~\lim_{g\rightarrow\infty}\prod_{j=1}^{n}\frac{1+yg^{k^{(i)}_{j}}}{1-g^{k^{(i)}_{j}}}=(-y)^{n-d_{i}}.\ee
Here $d_{i}$ is the number of the negative integers in $k_{1}^{(i)},
\cdots,k_{n}^{(i)}$. By (\ref{lim}), the RHS of (\ref{LOC}), and
hence the LHS of (\ref{LOC}), has well-defined limits at
$g=0,\infty$. Therefore $\chi_{y}(g,M)$ must be constant in $g$ and
$$\chi_{y}(g,M)\equiv\chi_{y}(id,M)=\chi_{y}(M),$$
which means
\be\label{identity}\chi_{y}(M)=\sum_{i=1}^{m}\prod_{j=1}^{n}\frac{1+yg^{k^{(i)}_{j}}}{1-g^{k^{(i)}_{j}}}\ee
holds for a dense subset of $S^{1}$ (the topological generators in
$S^{1}$ are dense). So this must be an identity in the indeterminate
$g$. Comparing the corresponding coefficients of (\ref{identity}),
we have
$$\chi^{p}(M)=\sum_{i=1}^{m}\frac{e_{p}(g^{k^{(i)}_{1}},\cdots,g^{k^{(i)}_{n}})}{\prod_{j=1}^{n}(1-g^{k_{j}^{(i)}})},~~0\leq p\leq n.$$
This completes the proof of Theorem \ref{main theorem}.

\begin{remark}
\begin{enumerate}
\item
Let $N_{P}$ denote the number of fixed points of the circle action
with exactly $p$ negative weights. From the proof we know that
\be\label{LOC2}
\begin{split}
\chi_{y}(M)&
=\sum_{i=1}^{m}(-y)^{d_{i}}=\sum_{i=1}^{m}(-y)^{n-d_{i}}\\
& =\sum_{p=0}^{n}N_{p}(-y)^{p}=\sum_{p=0}^{n}N_{p}(-y)^{n-p}.
\end{split}\ee
Hence\be\label{relation}\chi^{p}(M)=(-1)^{p}N_{p}=(-1)^{p}N_{n-p}=(-1)^{n}\chi^{n-p}(M).\ee

\item
As pointed out in Section 5.7 of \cite{HBJ}, this idea is
essentially due to Lusztig \cite{Lu}, Kosniowski \cite{Ko} and
Atiyah-Hirzebruch \cite{AH}. The former two used it to derive the
localization formula (\ref{LOC2}) for complex manifolds. Atiyah and
Hirzebruch used this idea in \cite{AH} to get their famous vanishing
theorem of $\hat{A}$-genus on spin manifolds.

\item
The localization formula (\ref{LOC2}) has been generalized to more
general cases (cf. \cite{HT} and \cite{KY}) since its first
appearance in \cite{Ko} on complex manifolds. To the author's
knowledge, in the case of almost-complex manifolds, the relations
between $\chi^{p}(M)$ and the weights as in Theorem \ref{main
theorem} have not been explicitly pointed out before. In fact, we
will see in the next section that, with these relations set up, we
can give many applications in related areas.
\end{enumerate}
\end{remark}
\section{Applications}
\subsection{}
Since the appearance of \cite{Mc}, symplectic circle actions on
six-dimensional symplectic manifolds have received a great deal of
attention due to the rich structures and possibilities in this
dimension. In \cite{Go}, Godinho showed that certain symplectic
circle actions on six-dimensional manifolds must be Hamiltonian.
More precisely, let the circle act symplectically on a
six-dimensional compact, connected, symplectic manifold with
non-empty isolated fixed points whose isotropy weights are of the
form $(\pm k_{1},\pm k_{2},\pm k_{3})$ for fixed integers $k_{1}\geq
k_{2}\geq k_{3}\geq 1$. Let $N_{0}$ and $N_{3}$ be the numbers of
the fixed points with the weights $(k_{1},k_{2},k_{3})$ and
$(-k_{1},-k_{2},-k_{3})$ respectively. Let $s_{1}, s_{2}, s_{3},
t_{1}, t_{2}$ and $t_{3}$ be the numbers of the weights
$(-k_{1},k_{2},k_{3}),$ $(k_{1},-k_{2},k_{3})$,
$(k_{1},k_{2},-k_{3})$, $(k_{1},-k_{2},-k_{3})$,
$(-k_{1},k_{2},-k_{3})$ and $(-k_{1},-k_{2},k_{3})$ respectively.
The following result is an extension of Godinho's results (cf.
\cite{Go}, Theorems 1.1, 1.2, 3.2 and 3.4)
\begin{theorem}
Let the circle act symplectically on a six-dimensional compact,
connected, symplectic manifold with non-empty isolated fixed points
and let the notations be as above, then there are exactly two
possibilities:
\begin{enumerate}
\item
$N_{0}=N_{3}=s_{2}=s_{3}=t_{2}=t_{3}=0$, $s_{1}=t_{1}\geq 1$ and
$k_{1}=k_{2}+k_{3}$,

in which case the action is non-Hamiltonian;

\item
$N_{0}=N_{3}=1$ and \be\begin{split} &
(g^{k_{3}}+g^{k_{2}}+g^{k_{1}})+(t_{1}g^{k_{2}+k_{3}}+t_{2}g^{k_{1}+k_{3}}+t_{3}g^{k_{1}+k_{2}})\\
=&(s_{3}g^{k_{3}}+s_{2}g^{k_{2}}+s_{1}g^{k_{1}})+(g^{k_{2}+k_{3}}+g^{k_{1}+k_{3}}+g^{k_{1}+k_{2}}),
\end{split}\nonumber\ee
in which case the action is Hamiltonian. Here of
course $g$ is an indeterminate.
\end{enumerate}
\end{theorem}

\begin{remark}
When $k_{1}\neq k_{2}+k_{3}$, case (1) will
lead to (\cite{Go}, Theorems 1.1 and 3.2) and case (2) will lead to
(\cite{Go}, Theorems 1.2 and 3.4).
\end{remark}

\begin{proof}
Note that we have the following expression for the Todd class
$\chi^{0}(M)$:
$$\chi^{0}(M)=\frac{N_{0}-(s_{1}g^{k_{1}}+s_{2}g^{k_{2}}+s_{3}g^{k_{3}})+(t_{1}g^{k_{2}+k_{3}}+t_{2}g^{k_{1}+k_{3}}+t_{3}g^{k_{1}+k_{2}})-N_{3}g^{k_{1}+k_{2}+k_{3}}}{(1-g^{k_{1}})(1-g^{k_{2}})(1-g^{k_{3}})}.$$
From Theorem 2.2 we have either
$$N_{0}-(s_{1}g^{k_{1}}+s_{2}g^{k_{2}}+s_{3}g^{k_{3}})+(t_{1}g^{k_{2}+k_{3}}+t_{2}g^{k_{1}+k_{3}}+t_{3}g^{k_{1}+k_{2}})-N_{3}g^{k_{1}+k_{2}+k_{3}}=0,$$
in which case the action is non-Hamiltonian, or \be\begin{split}
& N_{0}-(s_{1}g^{k_{1}}+s_{2}g^{k_{2}}+s_{3}g^{k_{3}})+(t_{1}g^{k_{2}+k_{3}}+t_{2}g^{k_{1}+k_{3}}+t_{3}g^{k_{1}+k_{2}})-N_{3}g^{k_{1}+k_{2}+k_{3}}\\
=& (1-g^{k_{1}})(1-g^{k_{2}})(1-g^{k_{3}}),
\end{split}\nonumber\ee
 in which case the action is Hamiltonian.
\end{proof}

We can also reproduce Tolman-Weitsman's following result (\cite{TW},
Theorem 1).

\begin{theorem}[Tolman-Weitsman]
Let $M^{2n}$ be a compact, connected symplectic manifold, equipped
with a semifree symplectic circle action with non-empty isolated
fixed points. Then this circle action must be Hamiltonian.
\end{theorem}
\begin{proof}
Since the action is semifree, all the integer weights $k_{j}^{(i)}$
are $\pm 1$. We still use $N_{p}$ to denote the number of fixed
points with exactly $p$ negative weights as in Section 2. Hence
$$\chi^{0}(M)=\frac{\sum_{p=0}^{n}N_{p}(-g)^{p}}{(1-g)^{n}}.$$
By assumption, the fixed points are non-empty, hence at least one of
$N_{0},\cdots,N_{n}$ is nonzero, which means $\chi^{0}(M)\neq 0$.
According to Theorem 2.2, $\chi^{0}(M)=1$ and the action is
Hamiltonian. Moreover, $N_{p}={n\choose p},~0\leq p\leq n.$ (Compare
to Lemma 3.1 in \cite{TW}.)
\end{proof}

\begin{remark}
This result was also reproved by Fel'dman
(\cite{Fe}, Corollary 1), whose main tools are the Conner-Floyd
equations for the Todd class.
\end{remark}
\subsection{}
We still use the notations in Section 2. In this subsection we will
prove the following result.
\begin{theorem}\label{main theorem in section 3}
Suppose the almost-complex manifold $(M^{2n},J)$ admits a circle
action preserving the almost-complex structure $J$ and having
isolated fixed points. Then for any $k\in\mathbb{Z}-\{0\},$
$$\sharp\{k_{j}^{(i)}=k~|~1\leq i\leq m,~1\leq j\leq n\}=\sharp\{k_{j}^{(i)}=-k~|~1\leq i\leq m,~1\leq j\leq n\}.$$
Here $\sharp$ denotes the cardinality of a set.
\end{theorem}
The following corollary is a generalization of (\cite{PT}, Theorem
2) from the case of symplectic circle actions on symplectic
manifolds to that of circle actions on almost-complex manifolds.

\begin{corollary}
Suppose the almost-complex manifold $(M^{2n},J)$ admit a circle
action preserving the almost-complex structure $J$ and having
isolated fixed points. Then
$$\sum_{i=1}^{m}\sum_{j=1}^{n}k_{j}^{(i)}=0.$$
\end{corollary}

In order to derive Theorem \ref{main theorem in section 3}, we need
the following lemma.

\begin{lemma}\label{lemma}
$\forall$ $0\leq k\leq n$, we have
$$e_{k}(x_{1}+1,x_{2}+1,\cdots,x_{n}+1)=\sum_{i=0}^{k}{n-i\choose n-k}e_{i}(x_{1},x_{2},\cdots,x_{n}).$$
\end{lemma}
\begin{proof}
\be\begin{split} e_{k}(x_{1}+1,x_{2}+1,\cdots,x_{n}+1)&=\sum_{1\leq
i_{1}<\cdots<i_{k}\leq
n}(x_{i_{1}}+1)(x_{i_{2}}+1)\cdots(x_{i_{k}}+1)\\
& :=\sum_{i=0}^{k}a_{i}\cdot e_{i}(x_{1},x_{2},\cdots,x_{n}).
\end{split}\nonumber\ee It suffices to determine the coefficients $a_{i}$. Note
that there are ${n\choose k}{k\choose i}$ monomials of degree $i$ in
the expression $\sum_{1\leq i_{1}<\cdots<i_{k}\leq
n}(x_{i_{1}}+1)(x_{i_{2}}+1)\cdots(x_{i_{k}}+1)$. Hence,
$$a_{i}=\frac{{n\choose
k}{k\choose i}}{{n\choose i}}={n-i\choose n-k}$$ as
$e_{i}(x_{1},\cdots,x_{n})$ has ${n\choose i}$ monomials.
\end{proof}

\emph{Proof of Theorem \ref{main theorem in section 3}.}

\be\begin{split}
\sum_{i=1}^{m}\sum_{j=1}^{n}\frac{1}{1-g^{k_{j}^{(i)}}}&=\sum_{i=1}^{m}\frac{e_{n-1}(1-g^{k_{1}^{(i)}},\cdots,1-g^{k_{n}^{(i)}})}{\prod_{j=1}^{n}(1-g^{k_{j}^{(i)}})}\\
&=\sum_{l=0}^{n-1}(n-l)(-1)^{l}\sum_{i=1}^{m}\frac{e_{l}(g^{k_{1}^{(i)}},\cdots,g^{k_{n}^{(i)}})}{\prod_{j=1}^{n}(1-g^{k_{j}^{(i)}})}~~(\textrm{by
Lemma \ref{lemma}})\\
&=\sum_{l=0}^{n-1}(n-l)(-1)^{l}\chi^{l}(M)\\
&=\sum_{l=0}^{n-1}(n-l)N_{l}\\
&=\frac{n}{2}\sum^{n}_{l=0}N_{l}~~(\textrm{by}~N_{l}=N_{n-l})\\
&=\frac{mn}{2}.
\end{split}\nonumber\ee
Note that for any $k\in\mathbb{Z}-\{0\},$
$\frac{1}{1-g^{k}}+\frac{1}{1-g^{-k}}=1$. So what we have showed
implies that, for any $k\in\mathbb{Z}-\{0\},$
$$\sharp\{k_{j}^{(i)}=k~|~1\leq i\leq m,~1\leq j\leq n\}=\sharp\{k_{j}^{(i)}=-k~|~1\leq i\leq m,~1\leq j\leq n\}.$$

\bibliographystyle{amsplain}
{\bf Acknowledgements.}~The author thanks the referee for his/her
very careful reading of the earlier version of this paper and many
fruitful comments and suggestions, which improve the quality of this
paper.

\end{document}